\documentclass[11pt,reqno]{amsart}
\usepackage{fullpage}
\usepackage{a4wide}
\usepackage[utf8x]{inputenc}
\usepackage[T1]{fontenc}
\usepackage[english]{babel}
\usepackage{hyperref}
\usepackage{amsfonts,amsmath,amssymb,amsthm,amsrefs}
\usepackage{latexsym}
\usepackage{layout}
\usepackage{dsfont}
\usepackage{color}

\newtheorem{theorem}{Theorem}

\newtheorem{lemma}[theorem]{Lemma}
\newtheorem{corol}[theorem]{Corollary}

\theoremstyle{definition}

\theoremstyle{remark}
\newtheorem{remark}[theorem]{Remark}

\newcommand{\kappaq}{\kappa^{q}}
\newcommand{\p}{\mathbb{P}}
\newcommand{\e}{\mathbb{E}}

\newcommand{\reals}{\mathbb{R}}
\newcommand{\ind}{\mathbf{1}}

\newcommand{\me}{\mathrm{e}}
\newcommand{\md}{\mathrm{d}}
\newcommand{\R}{\mathcal{F}}

\newcommand{\drift}{c}

\def\beq{\begin{eqnarray}} \def\eeq{\end{eqnarray}}

\def\al*#1{\begin{align*}#1\end{align*}}

\def\ga*#1{\begin{gather*}#1\end{gather*}}

\def\alat*#1#2{\begin{alignat*}{#1}#2\end{alignat*}}
\def\bea{\begin{eqnarray*}}
\def\eea{\end{eqnarray*}}
\def\ml*#1{\begin{multline*}#1\end{multline*}}

\allowdisplaybreaks

\begin{document}

\title[]{A unified approach to ruin probabilities with delays for spectrally negative L\'evy processes}
\author[]{Mohamed Amine Lkabous}
\address{D\'epartement de math\'ematiques, Universit\'e du Qu\'ebec \`a Montr\'eal (UQAM), 201 av.\ Pr\'esident-Kennedy, Montr\'eal (Qu\'ebec) H2X 3Y7, Canada}
\email{lkabous.mohamed\_amine@courrier.uqam.ca}
\author[]{Jean-Fran\c{c}ois Renaud}
\address{D\'epartement de math\'ematiques, Universit\'e du Qu\'ebec \`a Montr\'eal (UQAM), 201 av.\ Pr\'esident-Kennedy, Montr\'eal (Qu\'ebec) H2X 3Y7, Canada}
\email{renaud.jf@uqam.ca}

\date{\today}

\begin{abstract}
In this paper, we unify two popular approaches for the definition of actuarial ruin with implementation delays, also known as Parisian ruin. Our new definition of ruin includes both deterministic delays and exponentially distributed delays: ruin is declared the first time an excursion in the red zone lasts longer than an implementation delay with a deterministic and a stochastic component. For this Parisian ruin with mixed delays, we identify the joint distribution of the time of ruin and the deficit at ruin, therefore providing generalizations of many results previously obtained, such as in \cite{baurdoux_et_al_2015} and \cite{loeffenetal2017} for the case of an exponential delay and that of a deterministic delay, respectively.
\end{abstract}

\keywords{Ruin with delays, Parisian ruin, Lévy insurance risk processes.}
\subjclass[2000]{60G51, 91B30}
\maketitle

\section{Introduction}

%

In Parisian ruin models, if the insurance company defaults, then it is granted time to recover before liquidation. More precisely, Parisian ruin occurs if the time spent below a pre-determined critical level (default level) is longer than a delay, also called the \textit{clock}. Originally, two types of Parisian ruin have been considered, one with deterministic delays (see e.g.\ \cites{czarnapalmowski2010, loeffenetal2013,Wong_Cheung2015}) and another one with stochastic delays (\cites{landriaultetal2011,landriaultetal2014,baurdoux_et_al_2015}). These two types of Parisian ruin start a new clock each time the surplus enters the \textit{red zone}, either deterministic or stochastic. Recently, other definitions of Parisian ruin have been proposed; see e.g.\ \cites{guerinrenaud2015,lietal2017,lkabous2018}.

In this paper, we unify the definitions of Parisian ruin with deterministic delays and Parisian ruin with exponentially distributed delays by considering \textit{mixed delays}. Indeed, for this unified version of Parisian ruin, the race is between the duration of an excursion in the red zone, a deterministic implementation delay $r>0$ and a random delay described by an exponential random variable with rate $q>0$. More precisely, ruin occurs the first time an excursion below zero lasts longer than one of the two delays. Our main contributions are generalizations of several recent results obtained by Loeffen \textit{et al}. \cite{loeffenetal2017} and Lkabous \textit{et al}. \cite{lkabousetal2016}. The identities involve \textit{second-generation scale functions} and also the distribution of the spectrally negative L\'evy process at a fixed time. As they have a similar structure as the ones in \cite{lkabousetal2016}, \cite{loeffenetal2013} and \cite{loeffenetal2017}, we can then analyze limiting cases in order to recover previous results related to other definitions of Parisian ruin.

\subsection{A unified approach to Parisian ruin}

For a standard L\'evy insurance risk process $X$, the time of Parisian ruin, with delay $r>0$, has been studied first in \cite{loeffenetal2013}: it is defined as 
\begin{equation}\label{kappar}
\kappa_r = \inf \left\lbrace t > 0 \colon t - g_t > r \right\rbrace ,
\end{equation}
where $g_t = \sup \left\lbrace 0 \leq s \leq t \colon X_s \geq 0\right\rbrace$. Then, Parisian ruin occurs the first time an excursion below
zero lasts longer than the fixed implementation delay $r$. 

Parisian ruin with stochastic delays has been considered in \cites{landriaultetal2011,landriaultetal2014,baurdoux_et_al_2015,albrecheretal2016}. In this definition of ruin, the fixed delay is replaced by an independent exponential random. This time of ruin is defined as
\begin{equation*}
\kappa^{q}=\inf \left\{ t>0 \colon t-g_{t}>\me^{g_{t}}_{q} \right\} ,
\end{equation*}
where $\me^{g_{t}}_{q}$ is exponentially distributed with rate $q>0$. It is denoted by $T_0^-$ in \cite{albrecheretal2016}.

For our new definition of Parisian ruin, the time of ruin is defined as
\begin{equation}  \label{kqr}
\kappa_{r}^{q}=\kappa^{q} \wedge \kappa_{r}= \inf \left\{ t>0 \colon t-g_{t} > \left( \me^{g_{t}}_{q} \wedge r \right) \right\} .
\end{equation}


The rest of the paper is organized as follows. In Section 2, we present the necessary background material on spectrally negative L\'evy processes and scale functions, including some fluctuation identities with delays already available in the literature. The main results are presented in Section 3, followed by a discussion on those results. In Section 4, we provide explicit computations of the probability of Parisian ruin with mixed delays for two specific Lévy risk processes. Finally, in Section 5, we derive new technical identities and then provide proofs for the main results.

\section{Spectrally negative L\'{e}vy processes}\label{sectionSNLP}

We say that $X=\{X_t,t\geq 0\}$ is a L\'{e}vy insurance risk process if it is a spectrally negative L\'evy process (SNLP) on the filtered probability space $(\Omega,\mathcal{F},\{\mathcal{F}_t , t\geq0\}, \mathbb{P})$, that is
a process with stationary and independent increments and no positive jumps. To avoid trivialities, we exclude the case where $X$ has monotone paths.

As the L\'{e}vy process $X$ has no positive jumps, its Laplace transform exists: for all $\theta, t \geq 0$, 
\begin{equation*}
\e \left[ \mathrm{e}^{\theta X_t} \right] = \mathrm{e}^{t \psi(\theta)} ,
\end{equation*}
where 
\begin{equation*}
\psi(\theta) = \gamma \theta + \frac{1}{2} \sigma^2 \theta^2 + \int^{\infty}_0 \left( \mathrm{e}^{-\theta z} - 1 + \theta z \ind_{(0,1]}(z) \right) \Pi(\mathrm{d}z) ,
\end{equation*}
for $\gamma \in \reals$ and $\sigma \geq 0$, and where $\Pi$ is a $\sigma$-finite measure on $(0,\infty)$ such that 
\begin{equation*}
\int^{\infty}_0 (1 \wedge z^2) \Pi(\mathrm{d}z) < \infty .
\end{equation*}
This measure $\Pi$ is called the L\'{e}vy measure of $X$. We will use the standard Markovian notation: the law of $X$ when starting from $X_0 = x$ is denoted by $\p_x$ and the corresponding expectation by $\e_x$. We write $\p$ and $\e$
when $x=0$.

We also recall that there exists a function $\Phi \colon \lbrack 0,\infty) \rightarrow \lbrack 0,\infty)$ defined by $\Phi(p) = \sup \{\theta \geq 0 \colon \psi (\theta)=p\}$ (the right-inverse of $\psi$) such that 
\begin{equation*}
\psi \left(\Phi(p) \right) = p , \quad p \geq 0 .
\end{equation*}
We will write $\Phi=\Phi(p)$ when $p=0$. Note that we have $\Phi(p)=0$ if and only if $p=0$ and $\psi^{\prime
}(0+)\geq0$.


\subsection{Scale functions}

For $p \geq 0$, the $p$-scale function of $X$ is defined as the continuous function on $[0,\infty )$ with Laplace transform 
\begin{equation}\label{def_scale}
\int_{0}^{\infty} \mathrm{e}^{-\theta y} W^{(p)}(y) \mathrm{d}y = \frac{1}{\psi(\theta)-p} , \quad \text{for $\theta>\Phi(p)$.}
\end{equation}
This function is unique, positive and strictly increasing for $x\geq 0$ and is further continuous for $p \geq 0$. We extend $W^{(p)}$ to the whole real line by setting $W^{(p)}(x)=0$ for $x<0$. We will write $W=W^{(0)}$ when $p=0$.  

Using the notation introduced in \cite{albrecheretal2016}, we define another scale function $Z_p (x,\theta)$ by
\begin{equation}
Z_p (x,\theta) = \me^{\theta x} \left( 1-\left( \psi (\theta)-p \right) \int_{0}^{x} \me^{-\theta y} W^{(p)}(y) \mathrm{d}y \right) ,
\end{equation}
for $x\geq 0$, and by $Z_p (x,\theta)=\me^{\theta x}$, for $x<0$. We will write $Z=Z_{0}$ when $p=0$. Note that there is a one-to-one correspondance between $Z_p (x,\theta)$ just defined and the function $\mathcal{H}^{(p,s)}(x)$ defined in \cite{loeffenetal2014}; another version of this function was also defined in \cite{baurdoux_et_al_2015}. From now on we write $\psi_p(\theta)=\psi (\theta)-p$.


It is known that
\begin{equation}\label{CM1}
\lim_{b \rightarrow \infty} \dfrac{W^{(p)}(x+b)}{W^{(p)}(b)} = \me^{\Phi(p) x}
\end{equation}
and, for $\theta>\Phi(p)$,
\begin{equation}\label{limitZW}
\lim_{b \rightarrow \infty} \frac{Z_p (b,\theta)}{W^{(p)} (b)} = \frac{\psi_p(\theta)}{\theta-\Phi(p)} .
\end{equation}

For the sake of compactness of the results, and as it is now the way to go in the literature, we will use the following functions. For $p,p+s\geq 0$ and $a,x \in \reals$,
\begin{equation}\label{eq:convsnlp}
\mathcal{W}_{a}^{\left( p,s\right) }\left( x\right) = W^{\left(p+s\right)}\left( x\right) - s \int_{0}^{a} W^{\left( p+s \right) }\left( x-y\right) W^{\left( p\right) }\left( y\right) \md y .
\end{equation}
As obtained in \cite{loeffenetal2014}, we have
$$
(s-p) \int_0^x W^{(p)}(x-y) W^{(s)}(y) \mathrm{d}y = W^{(s)}(x) - W^{(p)}(x)
$$
and thus $\mathcal{W}_{a}^{(p,s)} (x)$ can also be written as follows:
\begin{equation}\label{w-script_second-def}
\mathcal{W}_{a}^{(p,s)} (x) = W^{(p)} (x) + s \int_a^x W^{(p+s)} (x-y) W^{(p)} (y) \md y .
\end{equation}
In what follows, we will use the two expressions.

For later use, note that we can show 
\begin{equation}\label{LapscriptW}
\int_{0}^{\infty} \me^{-\theta z} \mathcal{W}_{a}^{\left( p,s\right)} \left(a+z\right) \mathrm{d}z = \frac{Z_{p}\left( a,\theta\right)}{ \psi_{p+s}(\theta)}, \quad \theta >\Phi(p+s) .
\end{equation}

As most (classical) fluctuation identities can be expressed in terms of the above scale functions, we have collected some of those identities in Appendix~\ref{sect:classical_fluct_identities}.

\medskip

Inspired by \cite{loeffenetal2017}, we now define
\begin{equation}\label{Exp1}
\Lambda^{\left( p\right)} \left( x;r,s\right) = \int_{0}^{\infty } \mathcal{W}_{z}^{\left( p+s,-s\right)} \left( x+z\right) \frac{z}{r} \mathbb{P} \left(X_{r} \in \mathrm{d}z \right) .
\end{equation}
and using \eqref{eq:convsnlp}, we can rewrite $\Lambda^{\left( p\right)} \left( x;r,s\right)$ in the form
\begin{equation}\label{Exp2}
\Lambda^{\left( p\right)} \left( x;r,s\right)= \int_{0}^{\infty } \mathcal{W}_{x}^{\left( p,s\right)} \left( x+z\right) \frac{z}{r} \mathbb{P} \left(X_{r} \in \mathrm{d}z \right) .
\end{equation}
It could be named the \textit{$(s,r)$-delayed $p$-scale function of $X$}, as we will see in what follows (see Section~\ref{sect:discussion}). When $s=0$, we recover the function $\Lambda^{(p)}$ originally defined in \cite{loeffenetal2017}. To be more precise, when $s=0$, we have $\Lambda^{\left( p\right)} \left(x;r,0\right) = \Lambda^{\left( p\right)}\left( x,r\right)$, where
$$
\Lambda^{\left( p\right)}\left( x,r\right) = \int_{0}^{\infty} W^{\left( p\right)} \left( x+z\right) \frac{z}{r} \mathbb{P} \left( X_{r} \in \mathrm{d}z \right) ,
$$
and we write $\Lambda=\Lambda^{(0)}$.

Here is another connection between these two functions. We can re-write $\Lambda^{\left( p\right)}(x;r,s)$ using~\eqref{eq:convsnlp} and~\eqref{Exp2}:
\begin{equation}\label{lemma}
\Lambda^{\left(p\right)} \left(x;r,s\right) = \Lambda^{\left( p+s\right)} \left( x,r\right) - s \int_{0}^{x} \Lambda^{\left( p+s\right)} \left( y,r\right) W^{\left( p\right)} \left(x-y\right) \mathrm{d}y .
\end{equation}
This last identity will be very important in what follows.

Finally, in the main results, we will also use the following auxiliary function: for $x\in \reals$ and $p,p+s,\lambda\geq 0$, set
\begin{eqnarray*}
\R^{(p,\lambda)} (x;r,s) &=& \frac{1}{\psi_{p+s}(\lambda)}\left(\psi_p(\lambda) \me^{\psi_{p+s}(\lambda)r}-s \right) Z_{p} (x,\lambda) \\
&& \quad - \me^{\psi_{p+s}(\lambda)r}\psi_p \left( \lambda \right) \int_{0}^{r} \me^{-\psi \left( \lambda \right)u} \Lambda^{\left( p \right)} \left( x;u,s\right) \mathrm{d}u .
\end{eqnarray*}
For $\lambda=0$, we write $\R^{\left( p,0\right)} = \R^{\left( p\right)}$, where
$$
\R^{(p)}(x;r,s) = \frac{1}{\left(p+s \right)}\left(s+p \me^{-\left( p+s \right) r} \right) Z_{p} (x,0) + p \me^{-\left( p+s \right) r} \int_{0}^{r} \Lambda^{\left( p \right)} \left( x;u,s\right) \mathrm{d}u .
$$
and $s=0$, we have $\R^{\left( p,\lambda \right)} \left(x;r,0\right) = \R^{\left( p,\lambda \right)} \left(x,r\right)$.

\subsection{Fluctuation identities with delays}

Before moving on to the presentation of our main results, let us present some of the existing fluctuation identities with delays. Note that the expressions presented below might differ from the original ones as we will make use of the notation introduced in the previous section.

First, recall the definitions of standard first-passage stopping times (without delay): for $b \in \mathbb{R}$,
\begin{align*}
\tau_b^- &= \inf\{t>0 \colon X_t<b\} \quad \text{and} \quad \tau_b^+ = \inf\{t>0 \colon X_t > b\} ,
\end{align*}
with the convention $\inf \emptyset=\infty$.

For simplicity, let us assume that the \textit{net profit condition} is verified, i.e.\ $\e \left[ X_1 \right] = \psi^{\prime}(0+) > 0$.

The probability of Parisian ruin with exponential delays of rate $q>0$ was first computed in \cite{landriaultetal2011}: for $x \in \reals$, we have
\begin{equation}\label{Pruine2}
\mathbb{P}_{x} \left( \kappa^{q}<\infty \right) = 1 - \mathbb{E}\left[ X_{1}\right] \frac{\Phi(q)}{q} Z \left(x,\Phi(q) \right) .
\end{equation}
More general identities were later obtained in \cite{landriaultetal2014}, \cite{baurdoux_et_al_2015} and \cite{albrecheretal2016}. For example, for $x \leq b$ and $p,q\geq 0$, we have
\begin{equation}\label{Parexit3}
\e_{x} \left[ \me^{-p \kappaq} \ind_{\left\lbrace \kappaq < \tau _{b}^{+} \right\rbrace} \right] =\frac{q}{p+q}\left(Z_{p}\left(x,0\right) -\frac{Z_{p}\left( x,\Phi(p+q)\right)}{Z_{p}\left( b,\Phi(p+q)\right)} Z_{p} \left(b,0\right) \right)
\end{equation}
and
\begin{equation}\label{Parexit4}
\e_{x}\left[ \me^{-p \tau _{b}^{+}};\tau _{b}^{+}<\kappaq\right]=\frac{Z_{p}\left( x,\Phi(p+q)\right) }{Z_{p}\left( b,\Phi(p+q)\right)} ,
\end{equation}
where the first identity is taken from \cite{baurdoux_et_al_2015} (where a slightly different notation is used) and the second one is taken from \cite{albrecheretal2016}.

On the other hand, the probability of Parisian ruin with deterministic delays of size $r>0$ was first computed for an SNLP in \cite{loeffenetal2013}: for $x \in \reals$, we have
\begin{equation}\label{SNLPPr}
\p_{x} \left(\kappa_{r}<\infty \right) = 1 - \e[X_{1}] \frac{\Lambda(x,r)}{\int^{\infty}_{0} \dfrac{z}{r} \p \left( X_{r} \in \md z \right)} .
\end{equation}
Then, in \cite{lkabousetal2016} (see Theorem 4 in that paper with $\delta=0$) and in \cite{loeffenetal2017}, the following identity was obtained:
\begin{equation}\label{Parexit2}
\e_{x} \left[ \mathrm{e}^{-p \tau_{b}^{+}} \mathbf{1}_{\left\{ \tau_{b}^{+} < \kappa_{r} \right\}} \right] =\frac{\Lambda^{\left( p\right)} \left( x,r\right)}{\Lambda^{\left( p\right)} \left( b,r \right)} .
\end{equation}
Note that more general quantities were computed in \cite{loeffenetal2017}. We will make the connections later, when appropriate.

\section{Main results}

We are now ready to state our main results. They are generalizations of those presented in the previous section in the sense that $\kappa^q$ or $\kappa_r$ is replaced by the more general time of ruin $\kappa_r^q$.

First, here is the joint distribution of our new time of Parisian ruin and the corresponding deficit at ruin:
\begin{theorem}\label{maintheo}
For $p,\lambda \geq 0$, $b,q,r>0$ and $x \leq b$, we have
\begin{equation}\label{theo}
\mathbb{E}_{x} \left[ \mathrm{e}^{-p\kappa_{r}^{q} + \lambda X_{\kappa_{r}^{q}}} \mathbf{1}_{\left\{ \kappa_{r}^{q} < \tau_{b}^{+} \right\}} \right] = \R^{\left( p,\lambda \right)} \left( x;r,q\right) - \frac{\Lambda^{\left( p\right) }\left( x;r,q\right)}{\Lambda^{\left( p\right)} \left(b;r,q\right)} \R^{\left(p,\lambda \right)} \left( b;r,q \right)
\end{equation}
and
\begin{equation}\label{LaplaceTr3}
\mathbb{E}_{x} \left[ \mathrm{e}^{-p \tau_{b}^{+}} \mathbf{1}_{\left\{ \tau_{b}^{+} < \kappa_{r}^{q} \right\}} \right] =\frac{\Lambda^{\left( p\right) }\left( x;r,q\right)}{\Lambda^{\left( p\right) }\left( b;r,q\right)} .
\end{equation}

\end{theorem}

Setting $\lambda=0$ in the previous Theorem, we obtain the following Laplace transforms for the Parisian time of ruin:

\begin{corol}\label{LaplaceTr}
Let $p\geq 0$ and $b,q,r>0$. For $x \leq b$, we have
\begin{equation}\label{Laplace1}
\mathbb{E}_{x} \left[ \mathrm{e}^{-p \kappa_{r}^{q}} \mathbf{1}_{\left\{\kappa_{r}^{q} < \tau_{b}^{+} \right\}} \right] = \R^{\left(p\right)} \left( x;r,q \right) - \frac{\Lambda^{\left( p\right)} \left(x;r,q\right)}{\Lambda^{\left( p\right)} \left(b;r,q\right)} \R^{\left(p\right)} \left(b;r,q\right) .
\end{equation}
and, for $x\in \reals$, we have
\begin{equation}\label{LaplaceTr2}
\mathbb{E}_{x} \left[ \mathrm{e}^{-p \kappa_{r}^{q}} \mathbf{1}_{\left\{\kappa_{r}^{q} < \infty \right\}} \right] = \R^{\left(p\right)} \left( x;r,q \right)  - \Omega^{(p)}\left(r,q\right) \times \Lambda^{\left(p\right)} \left( x;r,q\right) 
\end{equation}
where 
$$
\Omega^{(p)}\left(r,q\right)= \frac{\frac{p}{\left( p+q\right) \Phi(p)}\left( q+p\me^{-\left( p+q\right)r}\right) + p \me^{-\left( p+q\right)r}\int_{0}^{r}\left(\int_{0}^{\infty } Z_{p+q}\left(z,\Phi(p)\right)\frac{z}{s}\mathbb{P}\left( X_{s}\in \mathrm{d}z\right)\right) \mathrm{d}s}{\int_{0}^{\infty} Z_{p+q} \left(z,\Phi(p)\right) \dfrac{z}{r} \mathbb{P}\left( X_{r} \in \mathrm{d}z \right)} .
$$

\end{corol}


Setting $p=0$ in~\eqref{LaplaceTr2}, we obtain the following expression for the probability of Parisian ruin with mixed delays:
\begin{corol}\label{cor:proba_ruin}
For $x \in \reals$ and $q,r>0$, we have
\begin{equation}\label{SNLPPr2}
\p_{x} \left( \kappa_{r}^{q} < \infty \right) = 1 - \left(\e\left[ X_{1}\right] \right)_+ \frac{\Lambda \left(x;r,q\right)}{\int_{0}^{\infty} Z_{q}(u,\Phi(0)) \dfrac{u}{r} \p\left( X_{r} \in \md u \right)} .
\end{equation}
\end{corol}
\subsection{Discussion on the results}\label{sect:discussion}

Our Parisian fluctuation identities are arguably compact and have a similar structure as classical fluctuation identities (without delays) as well as previously-obtained Parisian fluctuation identities (see e.g.\ \cites{loeffenetal2013,lkabousetal2016}).

Indeed, in Equation~\eqref{LaplaceTr3}, the $(q,r)$-delayed $p$-scale function $\Lambda^{\left( p\right)}(\cdot;r,q)$ plays a similar r\^ole as the one played by the classical $p$-scale function $W^{(p)}(\cdot)$ in the solution to the classical two-sided exit problem: for $x \leq b$, we have
\begin{equation}\label{exit1}
\mathbb{E}_{x} \left[ \mathrm{e}^{-p \tau_{b}^{+}} \mathbf{1}_{\left\{ \tau_{b}^{+} < \tau_0^- \right\}} \right] =\frac{W^{(p)}(x)}{W^{(p)}(b)} .
\end{equation}
Similarly, in Equation~\eqref{theo}, the $(q,r)$-delayed $(p,\lambda)$-scale function $\R^{\left( p,q,\lambda \right)} \left( \cdot,r\right)$ plays a similar r\^ole as the one played by the scale function $Z_p(\cdot,\lambda)$ in the following classical fluctuation identity: for $x \leq b$, we have
\begin{equation}\label{exit2}
\mathbb{E}_{x} \left[ \mathrm{e}^{-p \tau_0^- + \lambda X_{\tau_0^-}} \mathbf{1}_{\left\{ \tau_0^- < \tau_{b}^{+} \right\}} \right] = Z_p(x,\lambda) - \frac{W^{(p)}(x)}{W^{(p)}(b)} Z_p(b,\lambda) .
\end{equation}
See \cite{kyprianou2013} for the solution to the two-sided exit problem and see e.g.\ \cite{albrecheretal2016} for the latter identity.

\medskip


For the rest of this section, we will demonstrate that our results are simultaneously generalizing known identities for Parisian ruin with exponential delays and Parisian ruin with deterministic delays. As we have seen in the previous section, the results obtained so far in the literature, for either one definition of Parisian ruin or the other, did not seem to have strong connections allowing to recover the identity for one definition of ruin from the corresponding identity for the other definition of ruin.

We will also obtain analytical relationships between various intermediate quantities and auxiliary functions such as $\Lambda^{(p,q)}$ and $\R^{(p,q)}$.

%

\subsubsection{Parisian ruin with exponential delays}\label{rem}

First, we will show that the identity in~\eqref{LaplaceTr3} converges, as $r\rightarrow \infty$, to the solution of the delayed version of the two-sided exit problem when the implementation delay is exponentially distributed, namely the identity in~\eqref{Parexit4}.

Using~\eqref{Parexit2}, Lebesgue's convergence theorem and~\eqref{Laptrans}, we have
\begin{eqnarray*}
\lim_{r\rightarrow \infty} \frac{\Lambda^{\left( p+q\right)} \left( x,r\right)}{\Lambda^{\left( p+q\right)} \left(
b,r \right)} &=& \lim_{r\rightarrow \infty} \mathbb{E}_{x} \left[ \mathrm{e}^{-\left(p+q\right) \tau_{b}^{+}} \mathbf{1}_{\left\{ \tau_{b}^{+}<\kappa_{r} \right\}} \right] \notag \\
&=& \mathbb{E}_{x} \left[ \mathrm{e}^{-\left( p+q\right) \tau _{b}^{+}} \mathbf{1}_{\left\{ \tau_{b}^{+} < \infty \right\}} \right] = \me^{\Phi \left(p+q\right) \left( x-b\right)} .
\end{eqnarray*}
Consequently, using Lebesgue's convergence theorem and~\eqref{Exp2}:, we have
\begin{eqnarray*}
\lim_{r\rightarrow \infty} \frac{\Lambda^{\left( p\right)} \left( x;r,q\right)}{\Lambda^{\left( p\right)} \left( b;r,q\right)} &=& \lim_{r\rightarrow \infty} \frac{\Lambda^{\left( p+q\right)} \left(x,r\right)}{\Lambda^{\left( p+q \right)} \left( b,r\right)} - q \int_{0}^{x} W^{\left( p\right)} \left( x-u\right) \left( \lim_{r\rightarrow \infty} \frac{\Lambda^{\left(p+q\right)} \left( u,r\right)}{\Lambda^{\left( p+q\right)} \left(b,r\right)} \right) \mathrm{d}u \\
&=& \me^{\Phi \left( p+q\right) \left( x-b\right)} - q \int_{0}^{x} W^{\left( p\right)} \left( x-u\right) \me^{\Phi \left( p+q\right) \left( u-b\right)} \mathrm{d}u \\
&=& \me^{-\Phi \left( p+q\right) b} Z_{p}\left(x,\Phi(p+q)\right) .
\end{eqnarray*}
Finally, taking the limit as $r\rightarrow \infty$ of the identity in~\eqref{LaplaceTr3}, we get
\begin{eqnarray*}
\lim_{r\rightarrow \infty} \mathbb{E}_{x} \left[ \mathrm{e}^{-p\tau_{b}^{+}} \mathbf{1}_{\left\{ \tau_{b}^{+} < \kappa_{r}^{q} \right\} } \right] &=& \lim_{r\rightarrow \infty} \frac{\Lambda^{\left( p\right)} \left( x;r,q\right)}{\Lambda^{\left( p\right)} \left( b;r,q\right) } \\
&=& \lim_{r\rightarrow \infty} \frac{\Lambda^{\left( p\right)} \left( x;r,q\right) / \Lambda^{\left( p+q\right)} \left( b,r\right)}{\Lambda^{\left( p\right)} \left( b;r,q\right)/ \Lambda^{\left( p+q\right)} \left(b,r\right)} \\
&=& \frac{\me^{-\Phi \left( p+q\right) b} Z_{p} \left(x,\Phi(p+q) \right)}{\me^{-\Phi \left( p+q\right) b} Z_{p}\left(b,\Phi(p+q) \right)} \\
&=& \frac{Z_{p} \left(x,\Phi(p+q) \right)}{Z_{p} \left(b,\Phi(p+q) \right)} ,
\end{eqnarray*}
which is, as announced, the corresponding identity when there is no deterministic component in the delays; see~\eqref{Parexit4}.

\medskip

Second, we will show that the identity in~\eqref{Laplace1} converges, as $r\rightarrow \infty$, to the solution of the delayed version of the two-sided exit problem when the implementation delay is exponentially distributed, namely the identity in~\eqref{Parexit3}.

But before let us show that
\begin{equation}\label{eq:conv_Fscript}
\lim_{r \to \infty} \R^{\left( p,q\right)} (x,r) = \dfrac{q}{p+q} Z_p (x,0) .
\end{equation}
We want to compute the following limit:
\begin{eqnarray*}
\lim_{r\rightarrow \infty} \R^{\left( p,q\right)} \left( x,r\right) &=& \lim_{ r\rightarrow \infty} \frac{1}{p+q}\left( q+p \me^{-\left( p+q\right)r} \right) Z_p \left( x,0\right) \notag \\
&& \quad + \lim_{r\rightarrow \infty} p \me^{-\left( p+q\right) r} \int_{0}^{r} \Lambda^{\left(p\right)} \left( x;s,q\right) \mathrm{d}s \notag \\
&=& \frac{q}{p+q} Z_p (x,0) + \lim_{r\rightarrow \infty} p \me^{-\left( p+q\right) r} \int_{0}^{r} \Lambda^{\left( p \right)} \left(x;s,q\right) \mathrm{d}s .
\end{eqnarray*}
Recall Kendall's identity (see \cite[Corollary VII.3]{bertoin1996}): on $(0,\infty) \times (0,\infty)$, we have
\begin{equation*}
r \p(\tau_z^+ \in \mathrm{d}r) \mathrm{d}z = z \p(X_r \in \mathrm{d}z) \mathrm{d}r .
\end{equation*}
Using Kendall's identity and Tonelli's theorem, we have
\begin{equation*}
\int_{0}^{r} \Lambda^{\left( p\right)} \left( x;s,q\right) \mathrm{d}s = \int_{0}^{\infty} \mathcal{W}_{z}^{\left( p+q,-q \right)} \left( x+z\right) \mathbb{P}\left( \tau_{z}^{+} \leq r\right) \mathrm{d}z .
\end{equation*}
Taking Laplace transforms in $r$ on both sides, together with~\eqref{LapscriptW} and~\eqref{Laptrans}, yields
\begin{eqnarray*}
\int_{0}^{\infty} \me^{-\theta r} \left( \int_{0}^{r} \Lambda^{\left( p\right)} \left( x;s,q\right) \mathrm{d}s \right) \mathrm{d}r &=& \frac{1}{\theta} \int_{0}^{\infty} \me^{-\Phi(\theta) z} \mathcal{W}_{z}^{\left( p+q, -q \right)} \left( x+z\right) \mathrm{d}z \\
&=& \frac{Z_{p+q} \left(x,\Phi(\theta) \right)}{\theta \left( \theta-p \right)} .
\end{eqnarray*}
Then, using the Final value theorem, we obtain
\begin{equation*}
\lim_{r\rightarrow \infty} \int_{0}^{r} \Lambda^{\left( p \right)} \left(x;s,q\right) \mathrm{d}s = \lim_{\theta \rightarrow 0} \frac{Z_{p+q}\left(x,\Phi(\theta) \right)}{\left( \theta-p \right)} = \frac{-Z_{p+q} \left(x,\Phi(0) \right)}{p+q} .
\end{equation*}

To prove that the identity in~\eqref{Laplace1} converges, as $r\rightarrow \infty$, to the identity in~\eqref{Parexit3}, it suffices to use the fact that
$$
\lim_{r\rightarrow \infty} \frac{\Lambda^{\left( p\right)} \left( x;r,q\right)}{\Lambda^{\left( p\right)} \left( b;r,q\right) } = \frac{Z_{p} \left(x,\Phi(p+q) \right)}{Z_{p} \left(b,\Phi(p+q) \right)} ,
$$
as shown above, together with~\eqref{eq:conv_Fscript}.

\subsubsection{Parisian ruin with deterministic delays}

It is straightforward to verify that our results are generalizing known identities for Parisian ruin with deterministic delays.

Indeed, in identity~\eqref{theo} of Theorem~\ref{maintheo}, if we take the limit when $q \to 0$, then we get
$$
\lim_{q \to 0} \mathbb{E}_{x} \left[ \mathrm{e}^{-p\kappa_{r}^{q} + \lambda X_{\kappa_{r}^{q}}} \mathbf{1}_{\left\{ \kappa_{r}^{q} < \tau_{b}^{+} \right\}} \right] = \R^{\left( p,\lambda \right)} \left( x,r\right) - \frac{\Lambda^{\left( p\right) }\left( x,r\right)}{\Lambda^{\left( p\right)} \left(b,r\right)} \R^{\left( p,\lambda \right)} \left( b,r \right) ,
$$
as already obtained in \cite{loeffenetal2017}, with a slightly different notation.
\begin{remark}
Using the same techniques as in Subsection~\ref{rem}, we can also show that identities in Theorem~\ref{maintheo} and Corollary~\ref{cor:proba_ruin} converge to the identities \eqref{exit1}, \eqref{exit2} and \eqref{E:classicalruinprobaX} related to classical ruin.
\end{remark}
\section{Examples}

We now present two classical models for which we can compute easily the probability of mixed Parisian ruin, as given in Corollary~\ref{cor:proba_ruin}. Note that, to use the formula in~\eqref{SNLPPr2}, one needs to have an expression for the $0$-scale function $W$ and the distribution of the underlying Lévy risk process $X$.

We will also verify that we recover the known expressions for the probability of Parisian ruin with exponentially distributed delays, i.e.\ without the deterministic component.

\subsection{Brownian risk process}

Let $X$ be a Brownian risk processes, i.e.\
$$
X_t - X_0 = \drift t +\sigma B_t ,
$$
where $B=\{B_t, t\geq 0\}$ is a standard Brownian motion.

In this case, for $x\geq 0$ and $q>0$, the scale functions are given by
\begin{align*}
W(x) &= \frac{1}{c} \left( 1-\mathrm{e}^{-2\frac{c}{\sigma^2} x} \right) ,\\
W^{(q)} (x) &= \frac{1}{\sqrt{c^{2}+2q}} \left( \mathrm{e}^{x \left( \sqrt{c^{2}+2q}-c\right)} - \mathrm{e}^{-x \left( \sqrt{c^{2}+2q}+c\right) }\right) ,\\
Z_q (x,0) &= \frac{q}{\sqrt{c^{2}+2q}}\left( \frac{\mathrm{e}^{x\left( \sqrt{c^{2}+2q}-c\right) }}{\sqrt{c^{2}+2q}-c}+\frac{\mathrm{e}^{-x\left( \sqrt{c^{2}+2q}+c\right)}}{\sqrt{c^{2}+2q}+c}\right) ,\\
\mathcal{W}_{z}^{(q,-q)} (x) &= \frac{Z_q (x,0)}{c} \left(1 - \mathrm{e}^{-2cx} \right) .
\end{align*}

\medskip

Thus,
\begin{eqnarray*}
\mathcal{W}_{z}^{\left(q,-q\right) }\left( x+z\right) &=&\frac{\mathrm{e}^{z\left( \sqrt{c^{2}+2q}-c\right) }}{\sqrt{c^{2}+2q}-c}+\frac{\mathrm{e}^{-z\left( \sqrt{c^{2}+2q}+c\right) }}{\sqrt{c^{2}+2q}+c}-\frac{\mathrm{e}^{-2cx}\mathrm{e}^{z\left( \sqrt{c^{2}+2q}-c\right) }}{\sqrt{c^{2}+2q}+c}-\frac{\mathrm{e}^{-2cx}\mathrm{e}^{-z\left( \sqrt{c^{2}+2q}+c\right) }}{\sqrt{c^{2}+2q}-c} \\
&=& A_1(x) \mathrm{e}^{z\left( \sqrt{c^{2}+2q}-c\right)} + A_2(x) \mathrm{e}^{-z\left( \sqrt{c^{2}+2q}+c\right) },
\end{eqnarray*}
where 
\begin{eqnarray*}
A_1 (x) &=&\frac{q}{c\sqrt{c^{2}+2q}\left( \sqrt{c^{2}+2q}-c\right) }-\frac{q\mathrm{e}^{-2cx}}{c\sqrt{c^{2}+2q}\left( \sqrt{c^{2}+2q}+c\right) } , \\
A_2(x) &=&\frac{q}{c\sqrt{c^{2}+2q}\left( \sqrt{c^{2}+2q}+c\right) }-\frac{q\mathrm{e}^{-2cx}}{c\sqrt{c^{2}+2q}\left( \sqrt{c^{2}+2q}-c\right) }.
\end{eqnarray*}
First, we need to compute the following quantity 
\begin{eqnarray}\label{numerator}
\e\left[ X_{1}\right] \int_{0}^{\infty }\mathcal{W}_{z}^{\left(q,-q\right)}\left( x+z\right) \frac{z}{r}\mathbb{P}\left( X_{r} \in \mathrm{d}z\right) .
\end{eqnarray}
Making the change of variable $y=\left( z-r\sqrt{c^{2}+2q}\right) /\sqrt{r}$, we have 
\begin{eqnarray}\label{1}
&&\frac{1}{\sqrt{2\pi r}}\int_{0}^{\infty }\mathrm{e}^{z\left( \sqrt{c^{2}+2q}-c\right) }\frac{z}{r}\mathrm{e}^{-\left( z-cr\right) ^{2}/\left( 2r\right)}\mathrm{d}z  \notag \\
&=& \frac{1}{\sqrt{2r\pi }}\mathrm{e}^{-rc^{2}/2}+e\,^{rq} \sqrt{c^{2}+2q} \mathcal{N}\left( \sqrt{r}\sqrt{c^{2}+2q}\right)  \notag \\
&=& \Psi_{1} \left( c,r,q\right) ,
\end{eqnarray}
and, setting $y=-\left( z+r\sqrt{c^{2}+2q}\right) /\sqrt{r}$, we get 
\begin{eqnarray}\label{2}
&&\frac{1}{\sqrt{2\pi r}}\int_{0}^{\infty }\mathrm{e}^{-z\left( \sqrt{c^{2}+2q}+c\right) }\frac{z}{r}\mathrm{e}^{-\left( z-cr\right) ^{2}/\left(2r\right) }\mathrm{d}z \notag \\
&=& \frac{1}{\sqrt{2r\pi }}\mathrm{e}^{-rc^{2}/2}- \mathrm e^{rq} \sqrt{c^{2}+2q} \mathcal{N}\left( -\sqrt{r}\sqrt{c^{2}+2q}\right)  \notag \\
&=& \Psi_{2}\left( c,r,q\right) ,
\end{eqnarray}
where $\mathcal{N}$ is the cumulative distribution of the standard normal distribution. Using~\eqref{1} and~\eqref{2} in~\eqref{numerator}, we obtain 
\begin{eqnarray*}
\mathbb{E}\left[ X_{1}\right] \int_{0}^{\infty }\mathcal{W}_{z}^{\left(q,-q\right) } \left( x+z\right) \frac{z}{r}\mathbb{P} \left( X_{r}\in \mathrm{d}z \right) &=& \frac{1}{\sqrt{2\pi r}}\int_{0}^{\infty }\mathcal{W}_{z}^{\left(0,q\right) }\left( x+z\right) z\mathrm{e}^{-\left( z-cr\right) ^{2}/\left(2r\right) }\mathrm{d}z \\
&=& A_1(x) \Psi_{1} \left( c,r,q\right) + A_2(x) \Psi_{2}\left( c,r,q\right) ,
\end{eqnarray*}
while the denominator is given by 
\begin{eqnarray*}
\int_{0}^{\infty } Z_{q}(z,0) \frac{z}{r}\mathbb{P} \left( X_{r}\in \mathrm{d}z\right) &=&\frac{q}{\sqrt{c^{2}+2q}}%
\int_{0}^{\infty }\left( \frac{\mathrm{e}^{z\left( \sqrt{c^{2}+2q}-c\right) }}{\sqrt{c^{2}+2q}-c}+\frac{\mathrm{e}^{-z\left( \sqrt{c^{2}+2q}+c\right) }}{\sqrt{c^{2}+2q}+c}\right) \frac{z}{r}\mathbb{P}\left( X_{r}\in \mathrm{d}z \right) \\
&=&\frac{q}{\left( \sqrt{c^{2}+2q}-c\right) \sqrt{c^{2}+2q}}\Psi _{1}\left(c,r,q\right) \\
&&+\frac{q}{\left( \sqrt{c^{2}+2q}+c\right) \sqrt{c^{2}+2q}}\Psi _{2}\left(c,r,q\right) .
\end{eqnarray*}
Putting all the terms together, we get 
\begin{equation*}
\p_{x} \left( \kappa_{r}^{q} < \infty \right) = 1 - \frac{A_1(x) \Psi_{1} \left(c,r,q\right) + A_2(x) \Psi_{2} \left( c,r,q\right)}{\frac{q}{\left( \sqrt{c^{2}+2q}-c\right) \sqrt{c^{2}+2q}} \Psi_{1} \left( c,r,q\right) + \frac{q}{\left( 
\sqrt{c^{2}+2q} + c \right) \sqrt{c^{2}+2q}} \Psi_{2} \left( c,r,q\right) }.
\end{equation*}

Note that since
\begin{equation*}
\lim_{r \rightarrow \infty} \frac{\Psi _{1}\left( c,r,q\right)}{\mathrm e^{rq}}= \sqrt{c^{2}+2q}
\end{equation*}
and 
\begin{equation*}
\lim_{r \rightarrow \infty} \frac{\Psi _{2}\left( c,r,q\right) }{\mathrm e^{rq}} = 0 ,
\end{equation*}
we recover 
\begin{equation*}
\lim_{r\rightarrow \infty} \mathbb{P}_{x}\left( \kappa _{r}^{q}<\infty \right) = \frac{\mathrm e^{-2xc}\left( \sqrt{c^{2}+2q}-c\right) }{\sqrt{c^{2}+2q}+c} = \mathbb{P}_{x}\left( \kappa^{q}<\infty \right) ,
\end{equation*}
which is the probability of Parisian ruin with exponentially distributed delays, as given in~\eqref{Pruine2}, for the Brownian risk model.

\subsection{Cramér-Lundberg process with exponential claims}

Let $X$ be a Cramér-Lundberg risk processes with exponentially distributed claims, i.e.\
$$
X_t - X_0 = \drift t - \sum_{i=1}^{N_t} C_i ,
$$
where $N=\{N_t, t\geq 0\}$ is a Poisson process with intensity $\eta>0$, and where $\{C_1, C_2, \dots\}$ are independent and exponentially distributed random variables with parameter $\alpha$. The Poisson process and the random variables are mutually independent.

In this case, for $x\geq 0$ and $q>0$, the scale functions are given by
\begin{align*}
W^{(q)} (x) &= \frac{1}{c\left( \Phi(q)-\theta_{q} \right)} \left( \frac{\mathrm e^{\Phi(q)x}}{\alpha + \Phi(q)} - \frac{\mathrm e^{\theta_{q}x}}{\alpha +\theta _{q}}\right) ,\\
Z_{q} (x,0) & =\frac{1}{c} \left( \frac{\left(q-c \theta_{q}\right) \mathrm e^{\Phi(q)x} + \left( q-c \Phi (q) \right) \mathrm e^{\theta_{q}x}}{\Phi(q)-\theta _{q}}\right) ,
\end{align*}
where 
\begin{eqnarray*}
\Phi(q) &=& \frac{1}{2c}\left( p+\lambda -c\alpha +\sqrt{\Delta _{q}}\right) , \\
\theta_{q} &=& \frac{1}{2c}\left( q+\lambda -c\alpha -\sqrt{\Delta _{q}}\right) , \\
\Delta_{q} &=& \left( q+\lambda -c\alpha \right) ^{2}+4c\alpha q .
\end{eqnarray*}
Then, for $x\geq z$, we get
\begin{eqnarray*}
\mathcal{W}_{z}^{(p,q)}(x) &=& q \frac{\alpha +\Phi(p+q)}{\sqrt{\Delta_{p+q} \Delta _{p}}} \mathrm e^{\Phi(p+q) (x-z)} \left[\frac{\alpha + \Phi(p)}{\Phi(p+q) - \Phi(p)} \mathrm e^{\Phi(p) z} - \frac{\alpha + \theta_{p}}{\Phi(p+q) - \theta_{p}} \mathrm e^{\theta_{p} z} \right] \\
&& - q \frac{\alpha + \theta_{p+q}}{\sqrt{\Delta_{p+q} \Delta_{p}}} \mathrm e^{\theta_{p+q} (x-z)} \left[ \frac{\alpha + \Phi(p)}{\theta_{p+q} - \Phi(p)} \mathrm e^{\Phi(p) z} - \frac{\alpha + \theta_{p}}{\theta_{p+q} -\theta_{p}} \mathrm e^{\theta_{p} z}\right] .
\end{eqnarray*}
and 
\begin{eqnarray*}
\mathcal{W}_{z}^{(q,-q)} (x+z) &=& q \frac{\alpha}{\sqrt{\Delta_{0} \Delta_{q}}} \left[ \frac{\alpha}{\Phi(q)} \mathrm e^{\Phi(q) z} - \frac{\alpha + \theta_{q}}{\theta_{q}} \mathrm e^{\theta_{p} z} \right] \\
&& + q \frac{\alpha + \theta_{0}}{\sqrt{\Delta_{0} \Delta_{q}}} \mathrm e^{\theta_{0} x} \left[ \frac{\alpha + \Phi(q)}{\theta_{0} - \Phi(q)} \mathrm e^{\Phi(q) z} - \frac{\alpha + \theta_{q}}{\theta_{0} - \theta_{q}} \mathrm e^{\theta_{p} z} \right] .
\end{eqnarray*}

As noted in \cite{loeffenetal2013}, we have 
\begin{multline*}
\mathbb{P} \left( \sum_{i=1}^{N_r}C_i\in\mathrm{d}y \right) =
\sum_{k=0}^\infty \mathbb{P }\left( \sum_{i=0}^k C_i\in\mathrm{d}y \right) \mathbb{P}(N_r=k) \\
= \mathrm{e}^{-\eta r} \left( \delta_0(\mathrm{d}y) + \mathrm{e}^{-\alpha y} \sum_{m=0}^\infty \frac{ (\alpha \eta r)^{m+1}}{m!(m+1)!} y^{m} \mathrm{d}y \right) ,
\end{multline*}
where $\delta_0(\mathrm{d}y)$ is a Dirac mass at $0$. We also have
\begin{align*}
\int_{0}^{\infty} \mathrm{e}^{f(q)}{\small z} & \mathbb{P} \left(X_{r}\in \mathrm{d}z\right) \\
&= \int_{0}^{\infty }\mathrm{e}^{f(q)z}{\small z}\mathrm{e}^{-\eta r}\left( \delta _{0}\left( cr-\mathrm{d}z\right) +\mathrm{e}^{-\alpha \left( cr-z\right) }\sum_{m=0}^{\infty }\frac{\left( \alpha \eta r\right) ^{m+1}}{m!\left( m+1\right) !}\left( cr-z\right) ^{m}\mathrm{d}z \right) \\
&= \mathrm{e}^{\left(f(q)c-\eta \right) r}\left(1 + \sum_{m=0}^{\infty }\frac{\left( \alpha \eta r\right) ^{m+1}}{m!\left(m+1\right) !}\int_{0}^{cr}\mathrm{e}^{-\left( \alpha +f(q)\right)
y}\left( cr-z\right) ^{m}\mathrm{d}z\right) \\
&= \mathrm{e}^{\left( f(q)c-\eta \right) r} + \mathrm{e}^{\left(f(q)c-\eta \right) r} \sum_{m=0}^{\infty} \frac{\left( \alpha \eta r\right)^{m+1}}{m!\left( m+1\right)!} \left[ c r \Gamma \left( m+1,cr \left( \alpha + f(q)\right) \right) \right. \\
& \qquad \left. - \frac{1}{\left( \alpha + f(q)\right) }\Gamma \left( m+2,cr\left( \alpha +f(q)\right) \right) \right] ,
\end{align*}
where $\Gamma(a,x)=\int_0^x \mathrm{e}^{-t}t^{a-1}\mathrm{d}t$ is the incomplete gamma function and
$f(q)$ is equal to either $\Phi(q)$ or $\theta_q$.

Putting all the pieces together, we obtain an expression for the probability of Parisian ruin with mixed delays.

\section{Intermediate results and proofs}

Before presenting the proofs of the main results, we need a few intermediate lemmas.


Recall that, for $\theta,r,q>0$ and $y\geq0$, we have
\begin{equation}\label{eq:lemmapart1}
\Lambda^{(q)} (0,r)= \mathrm{e}^{qr} ,
\end{equation}
\begin{equation}\label{eq:lemmapart2}
\int_0^\infty \me^{-\theta r} \Lambda^{(q)}(-y,r) \md r = \frac{\mathrm{e}^{-\Phi(\theta)y}}{\theta-q}
\end{equation}
and
\begin{equation}\label{eq:lemmapart3}
\int_0^\infty \mathrm{e}^{-\theta r} \int_{y}^\infty \frac{z}{r} \mathbb P \left(X_r\in\mathrm{d}z \right) \mathrm{d}r = \frac{1}{\Phi(\theta)} \mathrm{e}^{-\Phi(\theta) y} .
\end{equation}
See e.g.\ \cite{lkabousetal2016} for proofs of the above three equations.

\medskip

The next two lemmas are the reasons our main results can be expressed explicitly in terms of scale functions.

To prove Lemma~\ref{mainlemma} below, we will need first to prove Lemma~\ref{lem2} which provides a solution to the \textit{race} between the mixed clock and the underlying process trying to get back above zero. Despite the similarities with \cite[Lemma 4.2 and Lemma 4.3]{loeffenetal2017}, we will take another direct and simple approach.

\begin{lemma}\label{lem2}
For $x\leq 0$, $p,\lambda \geq 0$ and $q,r>0,$ we have 
\begin{equation}\label{L1}
\e_{x}\left[ \mathrm{e}^{-p\tau _{0}^{+}} \mathbf{1}_{\left\{ \tau_{0}^{+} \leq \me_{q} \wedge r \right\} } \right] =\me^{-\left(p+q\right) r} \Lambda^{\left(p+q\right) } \left(x,r\right)
\end{equation}
and
\begin{eqnarray}\label{L33}
\e_{x}\left[ \mathrm{e}^{-p\left( \me_{q}\wedge r\right) +\lambda X_{
\me_{q}\wedge r}}\mathbf{1}_{\left\{ \tau _{0}^{+}>\me_{q}\wedge r\right\} }
\right]  &=&\frac{\me^{\lambda x}}{
\psi_{p+q}(\lambda)}\left(\psi_p\left( \lambda \right) \me^{\psi_{p+q}(\lambda)r}-q\right) \nonumber \\
&&-\me^{\psi_{p+s}(\lambda)r}\psi_p \left( \lambda \right) \int_{0}^{r}\me^{-\psi \left( \lambda \right) s}\Lambda ^{\left( p+q\right) }\left(x,s\right) 
\mathrm{d}s \notag \\
 &&-\me^{-\left( p+q\right)
r}\Lambda ^{\left( p+q\right) }\left( x,r\right),
\end{eqnarray}
where, in the case $\lambda = \Phi(p+q)$, the ratio $\frac{\psi_p\left( \lambda \right) \me^{\psi_{p+q}(\lambda)r}-q}{
\psi_{p+q}(\lambda)}$ is understood in the limiting sense, i.e.\
$$
\lim_{\lambda \rightarrow \Phi(p+q)}\frac{\psi_p\left( \lambda \right) \me^{\psi_{p+q}(\lambda)r}-q}{
\psi_{p+q}(\lambda)} = 1+qr .
$$
\end{lemma}

\begin{proof}
We can extract from the proof of \cite[Lemma 8]{lkabousetal2016}, the following identity:
$$
\e_x \left[ \me^{-p \tau _{0}^{+}} \ind_{\left\{ \tau _{0}^{+} \leq r \right\}} \right] = \int_{0}^{\infty} \me^{-p r} W^{(p)} \left( x+z\right) \frac{z}{r} \p(X_{r}\in \md z) .
$$
Note that the term on the right-hand side is equal to $\me^{-p r} \Lambda^{(p)} (x,r)$. Then, the result in Equation~\eqref{L1} follows from this last identity together with the independence between $\me_q$ and the underlying L\'evy risk process $X$.

For $\theta>0$, using the potential measure in~\eqref{potmeas}, we have
\begin{align*}
\int_{0}^{\infty }\me^{-\theta r} & \mathbb{E}_{x} \left[ \mathrm{e}^{-p\left( \me_{q}\wedge r\right) +\lambda X_{\me_{q}\wedge r}}\mathbf{1}_{\left\{ \tau_{0}^{+}>\me_{q}\wedge r\right\} }\right] \mathrm{d}r
= \frac{1}{\theta} \e_{x} \left[ \mathrm{e}^{-p\left( \me_{q}\wedge \me_{\theta }\right) + \lambda X_{\me_{q}\wedge \me_{\theta}}}\mathbf{1}_{\left\{ \tau _{0}^{+}>\me_{q}\wedge \me_{\theta }\right\} }\right]  \\
&= \me^{\Phi \left( p+q+\theta \right) x}\frac{\left( q+\theta \right) }{\theta }\int_{-\infty }^{0}\me^{\lambda y}W^{\left( p+q+\theta \right)}\left( -y\right) \mathrm{d}y - \frac{\left( q+\theta \right) }{\theta }\int_{-\infty }^{0}\me^{\lambda y}W^{\left( p+q+\theta \right) }\left( x-y\right) \mathrm{d}y \\
&= \left( \frac{\psi_p(\lambda)}{\theta \left(\psi_{p+q+\theta}(\lambda) \right)} - \frac{1}{\theta} \right) \times \left( \mathrm{e}^{\Phi(\theta+p+q) x} - \mathrm{e}^{\lambda x} \right) ,
\end{align*}
where, in the last equality, we used~\eqref{def_scale} for $\lambda>\Phi(p+q+\theta)$. Then, by Laplace inversion, where identity~\eqref{eq:lemmapart2} is needed, and further simplifications the result in~\eqref{L33} follows.
\end{proof}

\begin{lemma}\label{mainlemma}
For $x\in \reals$, $p,\lambda \geq 0$ and $b,q,r>0$, we have
\begin{multline}\label{L4}
\mathbb{E}_{x}\left[ \mathrm{e}^{-p\tau _{0}^{-}}\mathbb{E}_{X_{_{\tau_{0}^{-}}}}\left[ \mathrm{e}^{-p\tau _{0}^{+}}\mathbf{1}_{\left\{ \tau_{0}^{+}\leq \me_{q}\wedge r\right\} }\right] \mathbf{1}_{\left\{ \tau_{0}^{-}<\tau _{b}^{+}\right\} }\right] \\
= \me^{-\left( p+q\right) r}\left( \Lambda^{\left( p\right)} \left(x;r,q\right) - \frac{W^{\left( p\right)} \left( x\right)}{W^{\left( p\right)} \left( b\right)}\Lambda^{\left( p\right)} \left(b;r,q\right) \right)
\end{multline}
and
\begin{multline}\label{L5}
\mathbb{E}_{x}\left[ \mathrm{e}^{-p\tau _{0}^{-}}\mathbb{E}_{X_{_{\tau_{0}^{-}}}}\left[ \mathrm{e}^{-p\left( \me_{q}\wedge r\right) +\lambda X_{\me_{q}\wedge r}}\mathbf{1}_{\left\{ \tau _{0}^{+}>\me_{q}\wedge r\right\} } \right] \mathbf{1}_{\left\{ \tau _{0}^{-}<\tau _{b}^{+}\right\} }\right]  \\
= \left( \R^{\left( p,\lambda \right) }\left( x;r,\lambda\right) -\me^{-\left( p+q\right) r}\Lambda^{\left( p\right)} \left(x;r,q\right)\right) \\
- \frac{W^{\left( p\right) }\left( x\right) }{W^{\left( p\right)} \left( b\right) }\left( \R^{\left( p,\lambda \right) } \left(b;r,q\right) -\me^{-\left( p+q\right) r} \Lambda^{\left( p\right)} \left(b;r,q\right) \right) .
\end{multline}
\end{lemma}

\begin{proof}
The proof consists in using Lemma~\ref{lem2} to compute the inner expectations and then use the following relationship:
$$
\e_x \left[ \mathrm{e}^{-p \tau_0^-} \Lambda^{(p+q)} \left(X_{\tau_0^-},r \right) \ind_{\{\tau_0^- < \tau_b^+\}} \right] = \Lambda^{(p)} (x;r,q) - \frac{W^{(p)}(x)}{W^{(p)}(b)} \Lambda^{(p)} (b;r,q) .
$$
This is proved with the result in Equation~\eqref{id1}.

Identity~\eqref{exit1} is also needed to complete the proof of~\eqref{L5}. The details are left to the reader.
\end{proof}

\subsection{Proof of Theorem~\ref{maintheo}}

The steps of the proof of Theorem~\ref{maintheo} is based on the new Lemma~{mainlemma} together with standard probabilistic decompositions.

For $x<0$, using the strong Markov property of the fact that $X$ is skip-free upward, we get 
\begin{eqnarray}  \label{xNegb}
\mathbb{E}_{x}\left[ \mathrm{e}^{-p\kappa _{r}^{q}+\lambda X_{\kappa
_{r}^{q}}}\mathbf{1}_{\left\{ \kappa _{r}^{q}<\tau _{b}^{+}\right\} }\right]
&=&\mathbb{E}_{x}\left[ \mathrm{\ e}^{-p\left( \me_{q}\wedge r\right)
+\lambda X_{\me_{q}\wedge r}}\mathbf{1}_{\left\{ \tau _{0}^{+}>\me_{q}\wedge
r\right\} }\right]\notag \\
&+&\mathbb{E}_{x}\left[ \mathrm{e}^{-p\tau _{0}^{+}}\mathbf{1}_{\left\{ \tau
_{0 }^{+}\leq \me_{q}\wedge r\right\} }\right] \mathbb{E}\left[ 
\mathrm{e}^{-p\kappa _{r}^{q}+\lambda X_{\kappa _{r}^{q}}}\mathbf{1}%
_{\left\{ \kappa _{r}^{q}<\tau _{b}^{+}\right\} }\right].
\end{eqnarray}
Now, for $x\geq 0$, using again the strong Markov property, we get 
\begin{equation*}
\mathbb{E}_{x}\left[ \mathrm{e}^{-p\kappa _{r}^{q}+\lambda X_{\kappa
_{r}^{q}}}\mathbf{1}_{\left\{ \kappa _{r}^{q}<\tau _{b}^{+}\right\} }\right]
=\mathbb{E}_{x}\left[ \mathrm{e}^{-p\tau _{0}^{-}}\mathbb{E}_{X_{_{\tau
_{0}^{-}}}}\left[ \mathrm{e}^{-p\kappa _{r}^{q}+\lambda X_{\kappa _{r}^{q}}}%
\mathbf{1}_{\left\{ \kappa _{r}^{q}<\tau _{b}^{+}\right\} }\right] \mathbf{1}%
_{\left\{ \tau _{0}^{-}<\tau _{b}^{+}\right\} }\right] .
\end{equation*}
Injecting~\eqref{xNegb} in this last expectation, we obtain
\begin{align}\label{laplaceallx}
& \mathbb{E}_{x}\left[ \mathrm{e}^{-p\kappa _{r}^{q}+\lambda X_{\kappa
_{r}^{q}}}\mathbf{1}_{\left\{ \kappa _{r}^{q}<\tau _{b}^{+}\right\} }\right]
=\mathbb{E}_{x}\left[ \mathrm{e}^{-p\tau _{0}^{-}}\mathbb{E}_{X_{_{\tau
_{0}^{-}}}}\left[ \mathrm{e}^{-p\left( \me_{q}\wedge r\right) +\lambda X_{\me%
_{q}\wedge r}}\mathbf{1}_{\left\{ \tau _{0}^{+}>\me_{q}\wedge r\right\} }%
\right] \mathbf{1}_{\left\{ \tau _{0}^{-}<\tau _{b}^{+}\right\} }\right] 
\notag \\
& \qquad +\mathbb{E}_{x}\left[ \mathrm{e}^{-p\tau _{0}^{-}}\mathbb{E}%
_{X_{_{\tau _{0}^{-}}}}\left[ \mathrm{e}^{-p\tau _{0}^{+}}\mathbf{1}%
_{\left\{ \tau _{0}^{+}\leq \me_{q}\wedge r\right\} }\right] \mathbf{1}%
_{\left\{ \tau _{0}^{-}<\tau _{b}^{+}\right\} }\right] \mathbb{E}\left[ 
\mathrm{e}^{-p\kappa _{r}^{q}+\lambda X_{\kappa _{r}^{q}}}\mathbf{1}%
_{\left\{ \kappa _{r}^{q}<\tau _{b}^{+}\right\} }\right].
\end{align}
Note that this decomposition holds for all $x \in \reals$.

\medskip

We will first prove the result in~\eqref{theo} for $x=0$. We split this part of the proof in two steps: for processes with paths of bounded variation (BV) and then for processes with paths of unbounded variation (UBV).

First, we assume that $X$ has paths of BV. Setting $x=0$ in~\eqref{laplaceallx} yields 
\begin{equation*}
\mathbb{E}\left[ \mathrm{e}^{-p\kappa _{r}^{q}+\lambda X_{\kappa _{r}^{q}}}%
\mathbf{1}_{\left\{ \kappa _{r}^{q}<\tau _{b}^{+}\right\} }\right] =\frac{%
\mathbb{E}\left[ \mathrm{e}^{-p\tau _{0}^{-}}\mathbb{E}_{X_{_{\tau
_{0}^{-}}}}\left[ \mathrm{e}^{-p\left( \me_{q}\wedge r\right) +\lambda X_{\me%
_{q}\wedge r}}\mathbf{1}_{\left\{ \tau _{0}^{+}>\me_{q}\wedge r\right\} }%
\right] \mathbf{1}_{\left\{ \tau _{0}^{-}<\tau _{b}^{+}\right\} }\right] }{1-%
\mathbb{E}\left[ \mathrm{e}^{-p\tau _{0}^{-}}\mathbb{E}_{X_{_{\tau
_{0}^{-}}}}\left[ \mathrm{e}^{-p\tau _{0}^{+}}\mathbf{1}_{\left\{ \tau
_{0}^{+}\leq \me_{q}\wedge r\right\} }\right] \mathbf{1}_{\left\{ \tau
_{0}^{-}<\tau _{b}^{+}\right\} }\right] } .
\end{equation*}
Using~\eqref{eq:lemmapart1} and~\eqref{L5}, the numerator can be written as
\begin{multline}
\mathbb{E}\left[ \mathrm{e}^{-p\tau _{0}^{-}}\mathbb{E}_{X_{_{\tau
_{0}^{-}}}}\left[ \mathrm{e}^{-p\left( \me_{q}\wedge r\right) +\lambda X_{\me%
_{q}\wedge r}}\mathbf{1}_{\left\{ \tau _{0}^{+}>\me_{q}\wedge r\right\} }%
\right] \mathbf{1}_{\left\{ \tau _{0}^{-}<\tau _{b}^{+}\right\} }\right]  \\
=-\frac{W^{\left( p\right) }\left( 0\right) }{W^{\left( p\right) }\left(
b\right) }\left( \R^{\left( p,\lambda \right) }\left( b;r,q\right) -%
\me^{-\left( p+q\right) r}\Lambda ^{\left( p\right) }\left( b;r,q\right)\right)
\end{multline}
while the denominator can be written as
\begin{multline}
1-\e\left[ \mathrm{e}^{-p\tau _{0}^{-}}\mathbb{E}_{X_{_{\tau
_{0}^{-}}}}\left[ \mathrm{e}^{-p\tau _{0}^{+}}\mathbf{1}_{\left\{ \tau
_{0}^{+}\leq \me_{q}\wedge r\right\} }\right] \mathbf{1}_{\left\{ \tau
_{0}^{-}<\tau _{b}^{+}\right\} }\right]  \\
=1-\me^{-\left( p+q\right) r}\left( \Lambda ^{\left( p\right) }\left(
0;r,q\right) -\frac{W^{\left( p\right) }\left( 0\right) }{W^{\left( p\right)
}\left( b\right) }\Lambda ^{\left( p\right) }\left(b;r,q\right) \right)  \\
= \frac{W^{\left( p\right) }\left( 0\right) }{W^{\left( p\right) } \left( b\right) } \me^{-\left( p+q\right) r} \Lambda^{\left( p\right)} \left( b;r,q \right) ,
\end{multline}
where in the last equality we used~\eqref{eq:lemmapart1}. Note that, since $X$ is assumed to be of BV, we have $W(0)>0$. Consequently, we have obtained
\begin{equation}\label{x=0}
\mathbb{E}\left[ \mathrm{e}^{-p\kappa _{r}^{q} +\lambda X_{\kappa _{r}^{q}}} \mathbf{1}_{\left\{ \kappa
_{r}^{q}<\tau _{b}^{+}\right\} }\right] =1-\frac{\R^{\left(
p,\lambda \right) }\left( b;r,q\right) }{\me^{-\left( p+q\right) r}\Lambda^{\left( p\right)} \left( b;r,q \right) }.
\end{equation}

Now, we assume $X$ has paths of UBV. Let us approximate the situation as follows (as in \cite{loeffenetal2013}). We denote by $\kappa_{r,\epsilon}^q$ the first time an excursion, starting when $X$ gets below zero and ending before $X$ gets back up to $\epsilon$, is longer than $\me_{q} \wedge r$. Mathematically,
\begin{equation*}
\kappa_{r,\epsilon }^{q}=\inf \left\{ t>0:t-g_{t}^{\epsilon } > \left( \me_{q}\wedge r\right) , X_{t-\left( \me_{q}\wedge r \right) } < 0 \right\} .
\end{equation*}
Using similar arguments as in the BV case, we can write
\begin{align*}
\mathbb{E}_{\epsilon} & \left[ \mathrm{e}^{-p\kappa _{r,\epsilon }^{q} + \lambda X_{\kappa _{r,\epsilon }^{q}}}\mathbf{1}_{\left\{ \kappa _{r,\epsilon}^{q} < \tau _{b}^{+} \right\}} \right]
= \frac{\mathbb{E}_{\epsilon}\left[ \mathrm{e}^{-p\tau _{0}^{-}}\mathbb{E}_{X_{_{\tau _{0}^{-}}}}\left[ \mathrm{e}%
^{-p\left( \me_{q}\wedge r\right) +\lambda X_{\me_{q}\wedge r}} \mathbf{1}_{\left\{ \tau _{\epsilon }^{+}>\me_{q}\wedge r\right\} }\right] \mathbf{1}_{\left\{ \tau _{0}^{-}<\tau _{b}^{+}\right\} }\right] }{1-\mathbb{E}_{\epsilon }\left[ \mathrm{e}^{-p\tau _{0}^{-}}\mathbb{E}_{X_{_{\tau_{0}^{-}}}}\left[ \mathrm{e}^{-p\tau _{\epsilon }^{+}}\mathbf{1}_{\left\{\tau _{\epsilon }^{+}\leq \me_{q}\wedge r\right\} }\right] \mathbf{1}_{\left\{ \tau _{0}^{-}<\tau _{b}^{+}\right\} }\right] } \\
&= \frac{\R_{\epsilon }^{\left( p,\lambda \right) }\left( \epsilon;r,q\right) - \mathrm e^{-\left( p+q\right) r}\Lambda_{\epsilon }^{\left( p\right)} \left( 0;r,q\right) - \frac{W^{\left( p\right)} \left( \epsilon \right)}{W^{\left( p\right)} \left( b\right)} \left( \R^{\left( p,\lambda\right)} \left( b;r,q\right) - \mathrm e^{-\left( p+q\right) r} \Lambda_{\epsilon}^{\left( p\right)} \left( b-\epsilon;r,q\right) \right)}{1 - \mathrm e^{-\left( p+q\right) r}\left( \Lambda_{\epsilon }^{\left( p\right)} \left( 0;r,q\right) - \frac{W^{\left( p\right)} \left( \epsilon \right) }{W^{\left( p\right) }\left( b\right) }\Lambda _{\epsilon }^{\left( p\right)} \left( b-\epsilon;r,q\right) \right) } ,
\end{align*}
where, from~\eqref{Exp1}, we define temporarily
\begin{equation*}
\Lambda_{\epsilon }^{\left( p\right)} \left( x,r,q\right) = \int_{\epsilon }^{\infty} \mathcal{W}_{z-\epsilon}^{\left(p+q,-q\right)} \left( x+z\right) \frac{z}{r} \mathbb{P} \left( X_{r} \in \mathrm{d}z \right) .
\end{equation*}

We will now compute the limit, as $\epsilon \rightarrow 0$, of the denominator and the numerator with an appropriate scaling. We can write
\begin{multline*}
\frac{1-\mathrm{e}^{-\left( p+q\right) r} \left( \Lambda_{\epsilon }^{\left( p\right) }\left( 0;r,r\right) -\frac{W^{\left( p\right) }\left( \epsilon \right) }{W^{\left( p\right) }\left(b\right) }\Lambda _{\epsilon }^{\left( p\right) }\left( b-\epsilon;r,q\right) \right)}{W^{\left( p\right) }\left( \epsilon \right)} \\
= \frac{1-\mathrm{e}^{-\left( p+q\right) r} \Lambda_{\epsilon }^{\left( p\right) } \left( 0;r,q\right)}{W^{\left(
p\right)} \left( \epsilon \right) } + \frac{\mathrm{e}^{-(p+q)r} \Lambda^{(p)}_\epsilon (b-\epsilon;r,q)}{W^{(p)} (b)} ,
\end{multline*}
where, using~\eqref{w-script_second-def}, we have
\begin{align*}
\frac{1-\mathrm{e}^{-\left( p+q\right) r} \Lambda_{\epsilon }^{\left( p\right) } \left( 0;r,q\right)}{W^{\left(
p\right)} \left( \epsilon \right) } &= \frac{1 - \mathrm{e}^{-(p+q)r} \int_\epsilon^\infty W^{(p+q)} (z) \frac{z}{r} \mathbb{P} \left( X_{r} \in \mathrm{d}z \right)}{W^{(p)}(\epsilon)} \\
& \qquad + q \mathrm e^{-(p+q)r} \frac{\int_{\epsilon}^{\infty} \left[ \int_{z-\epsilon}^{z} W^{(p)} (z-y) W^{(p+q)} (y) \md y \right] \frac{z}{r} \mathbb{P} \left( X_{r} \in \mathrm{d}z \right)}{W^{(p)}(\epsilon)} .
\end{align*}
We will show that this last expression converges to zero. First, using~\eqref{eq:lemmapart1} and then using the fact that $W^{(p+q)}$ is an increasing function, we can write
\begin{align*}
\frac{1 - \mathrm{e}^{-(p+q)r} \int_\epsilon^\infty W^{(p+q)} (z) \frac{z}{r} \mathbb{P} \left( X_{r} \in \mathrm{d}z \right)}{W^{(p)}(\epsilon)} &= \frac{\mathrm{e}^{-(p+q)r} \int_0^\epsilon W^{(p+q)} (z) \frac{z}{r} \mathbb{P} \left( X_{r} \in \mathrm{d}z \right)}{W^{(p)}(\epsilon)} \\
& \leq \frac{\mathrm{e}^{-(p+q) r}}{r} \frac{W^{(p+q)} (\epsilon)}{{W^{(p)}(\epsilon)}/\epsilon} \longrightarrow_{\epsilon \to 0} 0 ,
\end{align*}
since
\begin{equation*}
\lim_{\epsilon \rightarrow 0} \frac{W^{(p)}(\epsilon)}{\epsilon} =
\begin{cases}
\frac{2}{\sigma^2} & \text{if $\sigma >0$,} \\ 
\infty & \text{otherwise.}
\end{cases}
\end{equation*}
Similarly, using Lebesgue's convergence theorem, we can write
\begin{multline*}
\frac{\int_{\epsilon}^{\infty} \left[ \int_{z-\epsilon}^{z} W^{(p)} (z-y) W^{(p+q)} (y) \md y \right] \frac{z}{r} \mathbb{P} \left( X_{r} \in \mathrm{d}z \right)}{W^{(p)}(\epsilon)} \\
\leq \int_{0}^{\infty} \left[ \int_{z-\epsilon}^{z} W^{(p+q)} (y) \md y \right] \frac{z}{r} \mathbb{P} \left( X_{r} \in \mathrm{d}z \right) \longrightarrow_{\epsilon \to 0} 0 .
\end{multline*}
Therefore, we have obtained
$$
\lim_{\epsilon \to 0} \frac{1-\mathrm{e}^{-\left( p+q\right) r} \left( \Lambda_{\epsilon }^{\left( p\right) }\left( 0;r,q\right) -\frac{W^{\left( p\right) }\left( \epsilon \right) }{W^{\left( p\right) }\left(b\right) }\Lambda _{\epsilon }^{\left( p\right) }\left( b-\epsilon;r,q\right) \right)}{W^{\left( p\right) }\left( \epsilon \right)} = \frac{\mathrm{e}^{-(p+q)r} \Lambda^{(p)} (b;r,q)}{W^{(p)} (b)} .
$$

Using similar arguments, we can also show that
\begin{multline*}
\lim_{\epsilon \to 0} \frac{\R_{\epsilon}^{(p,\lambda)} (\epsilon;r,q) - \mathrm e^{-(p+q) r} \Lambda_{\epsilon}^{(p)} (0;r,q) - \frac{W^{(p)} (\epsilon)}{W^{(p)} (b)} \left( \R^{\left( p,\lambda\right)} \left( b;r,q\right) - \mathrm e^{-\left( p+q\right) r} \Lambda_{\epsilon}^{\left( p\right)} \left( b-\epsilon;r,q\right) \right)}{W^{(p)}(\epsilon)} \\
= \frac{\R^{\left( p,\lambda \right) }\left( b;r,q\right) - \mathrm e^{-\left( p+q\right) r} \Lambda ^{\left( p\right) }\left( b;r,q\right) }{W^{\left( p\right) }\left( b\right)} .
\end{multline*}
This concludes the proof for $x=0$.

Finally, no matter if $X$ is of BV or of UBV, using Equation~\eqref{laplaceallx}, Equation~\eqref{x=0} and identities in Lemma~\ref{mainlemma}, we can finish the proof of Theorem~\ref{maintheo}. To prove~\eqref{LaplaceTr3}, we can proceed as above. In both cases, the remaining details are left to the reader.

\subsection{Proof of Corollary~\ref{LaplaceTr}}


To deal with the limit as $b\rightarrow \infty $, we use \eqref{limitZW}, \eqref{CM1} and the fact that 
\begin{eqnarray*}
\lim_{b\rightarrow \infty }\frac{\mathcal{W}_{z}^{\left(p+q,-q\right)
}\left( b+z\right) }{W^{\left( p\right) }\left( b\right) } &=&\me^{\Phi(p)z }+q\lim_{b\rightarrow \infty }\int_{0}^{z}\frac{%
W^{\left( p\right) }\left( b+z-y\right) }{W^{\left( p\right) }\left(
b\right) }W^{\left( p+q\right) }\left( y\right) \md y\qquad  \\
&=&\me^{\Phi(p)z}+q\int_{0}^{z}\me^{\Phi(p)\left(
z-y\right) }W^{\left( p+q\right) }\left( y\right) \md y\\
&=&\me^{\Phi(p)z}\left( 1+q\int_{0}^{z}\me^{-\Phi
(p)y}W^{\left( p+q\right) }\left( y\right) \md y\right) \\
&=& Z_{p+q}\left(z,\Phi(p)\right) .
\end{eqnarray*}%
Then, 
\begin{eqnarray*}
&&\lim_{b\rightarrow \infty }\frac{\R^{\left( p\right) }\left(b;r,q\right) }{\Lambda ^{\left( p\right) }\left( b;r,q\right) }%
=\lim_{b\rightarrow \infty }\frac{\R^{\left( p\right) }\left(b;r,q\right) /W^{\left( p\right) }\left( b\right) }{\Lambda ^{\left(
p\right) }\left( b;r,q\right) /W^{\left( p\right) }\left( b\right) } \\
&=&\frac{\frac{p}{\left( p+q\right) \Phi(p)}\left( q+p\me^{-(p+q)r}\right)
+p\me^{-\left( p+q\right) r}\int_{0}^{r}\left(\int_{0}^{\infty }Z_{p+q}\left(z,\Phi
(p)\right)\frac{z}{s}\mathbb{P}\left( X_{s}\in \mathrm{d}z\right) \right) 
\mathrm{d}s}{\int_{0}^{\infty }Z_{p+q}\left(z,\Phi(p)\right) \frac{z}{r} \mathbb{P}\left( X_{r}\in \mathrm{d}z\right) }
.
\end{eqnarray*}

\appendix
\section{Fluctuation identities without delays}\label{sect:classical_fluct_identities}

Here is a collection of known fluctuation identities for the spectrally negative L\'{e}vy processes in terms of their scale functions.

For $p \geq 0$ and $a \leq x \leq b$ , we have 
\begin{equation}\label{Laptrans}
\mathbb{E}_{x} \left[ \mathrm{e}^{-p \tau_b^+} \mathbf{1}_{\{ \tau_b^{+} < \infty \}} \right] = \me^{-\Phi(p) (b-x)} .
\end{equation}
Moreover, the \textit{classical} probability of ruin is given by 
\begin{equation}  \label{E:classicalruinprobaX}
\p_x \left( \tau_0^- < \infty \right) = 1 - (\e \left[ X_1 \right])_{+}W(x).
\end{equation}
From \cite[Lemma 2.2]{loeffenetal2014}, we know that, for $p,s \geq 0$ and $a \leq x \leq b$, we have 
\begin{equation}\label{id1}
\e_x \left[ \me^{-p \tau_a^-} W^{(s)} \left( X_{\tau_a^-} \right) \ind_{\{ \tau_a^- < \tau_b^+ \}} \right] =
\mathcal{W}_a^{(s,p-s)} (x) - \frac{W^{(p)} (x-a)}{W^{(p)} (b-a)} \mathcal{W}_{a}^{(s,p-s)} (b) .
\end{equation}
Also, for $x \leq a$, we have the following expression for the $p$-potential measure of $X$ killed on exiting $(-\infty,a]$:
\begin{equation}\label{potmeas}
\int_0^{\infty} \me^{-pt} \mathbb{P}_{x} \left( X_{t}\in \mathrm{d}y , t < \tau _{a}^{+} \right) \md t = \left( 
\me^{\Phi(p)(x-a)} W^{\left( p\right)} \left( a-y\right) - W^{\left(p\right)} \left(x-y\right) \right) \mathrm{d}y .
\end{equation}

\section*{Acknowledgements}

Funding in support of this work was provided by the Natural Sciences and Engineering Research Council of Canada (NSERC).

M.\ A.\ Lkabous thanks the Institut des sciences math\'ematiques (ISM) and the Faculté des sciences at UQAM for their financial support (PhD scholarships).

\bibliographystyle{alpha}
\bibliography{LR2_janv2018}

\end{document}